\documentclass{amsart}

\newcommand{\C}{\mathbb{C}}

\newcommand{\R}{\mathbb{R}}
\newcommand{\N}{\mathbb{N}}

\newcommand{\cat}{^\frown}
\newcommand{\dom}{\operatorname{dom}}
\newcommand{\ran}{\operatorname{ran}}

\renewcommand{\Re}{\operatorname{Re}}
\renewcommand{\Im}{\operatorname{Im}}

\newcommand{\norm}[1]{\left\| #1 \right\|}

\newcommand{\B}{\mathcal{B}}
\newcommand{\zerovec}{\mathbf{0}}

\newcommand{\F}{\mathbb{F}}
\newcommand{\bfv}{\mathbf{v}}

\newcommand{\LB}{\mathcal{L}_{\operatorname{Banach}}}
\newcommand{\dsym}{\underline{d}}
\newcommand{\plus}{\bigoplus}

\newcommand{\LBL}{\mathcal{L}_{\operatorname{B-Lattice}}}
\newcommand{\meet}{\wedge}
\newcommand{\join}{\vee}
\newcommand{\TBL}{T_{\operatorname{B-Lattice}}}
\newcommand{\Ind}{\mathbf{1}}

\newcommand{\TLpR}{T_{L^p(\R)}}
\newcommand{\LLpC}{\mathcal{L}_{L^p(\C)}}
\newcommand{\TLpC}{T_{L^p(\C)}}

\theoremstyle{theorem}
\newtheorem{theorem}{Theorem}[section]
\newtheorem{lemma}[theorem]{Lemma}

\theoremstyle{definition}
\newtheorem{definition}[theorem]{Definition}

\theoremstyle{theorem}

\theoremstyle{theorem}

\theoremstyle{theorem}

\theoremstyle{theorem}

\theoremstyle{definition}

\theoremstyle{theorem}

\numberwithin{equation}{section}

\begin{document}
\title{Continuous logic and embeddings of Lebesgue spaces}
\author{Timothy H. McNicholl}
\address{Department of Mathematics\\
Iowa State University\\
Ames, Iowa 50011}
\email{mcnichol@iastate.edu}

\begin{abstract}
We use the compactness theorem of continuous logic to give a new proof that 
$L^r([0,1]; \R)$ isometrically embeds into $L^p([0,1]; \R)$ whenever $1 \leq p \leq r \leq 2$.  
We will also give a proof for the complex case.  This will involve a new characterization of complex $L^p$ spaces based on Banach lattices.
\end{abstract}
\maketitle

\section{Introduction}\label{sec:intro}

Let $\F$ denote either the field of real numbers or the field of complex numbers.  When $\Omega = (X, \mathcal{M}, \mu)$ is a measure space and $1 \leq p < \infty$, $L^p(\Omega;\F)$ denotes the set of all measurable $f : X \rightarrow \F$ with finite $p$-norm.  If $\mu$ is the counting measure on $\N$, we write $\ell^p(\F)$ for $L^p(\Omega; \F)$.  
We write $L^p(\F)$ for the class of all Banach spaces of the form $L^p(\Omega; \F)$ and $L^p$ for $L^p(\R) \cup L^p(\F)$.

When $\B_0$ and $\B_1$ are Banach spaces over $\F$, $\B_0$ is said to \emph{isometrically embed} into $\B_1$ if there is a linear isometric map $T : \B_0 \rightarrow \B_1$ (i.e. a linear map that preserves the norm).  For example, it is easy to show that $\ell^p(\F)$ isometrically embeds into $L^p([0,1];\F)$.  The question as to when does an $L^p(\F)$ space isometrically embed into an $L^r(\F)$ space for $r \neq p$ goes back to S. Banach's foundational treatise \emph{Theorie des Operations Lineaires} \cite{Banach.1987}.  In
that seminal work, Banach showed that $\ell^r(\F)$ does not isometrically embed into 
$L^p([0,1]; \F)$ if either $2 < p < r$ or $r < p < 2$.  Thus, $L^r([0,1]; \F)$ does not isometrically embed into $L^p([0,1]; \F)$ in these cases either.   In 1936, R. Paley showed that $\ell^r(\F)$ does not isometrically embed into $(L^p[0,1]; \F)$ 
if either $r,p$ are on opposite sides of $2$ or if $2 < r < p$ \cite{Paley.1936}.  These results were published posthumously, and the manuscript was assembled from Paley's notes by the editor F.J. Murray.  Murray noted that the proof in the remaining case $p < r < 2$ could not be reconstructed from Paley's notes.  

So, it was fairly surprising when in 1965, Bretagnolle, Dacunha-Castelle, and Krivine showed the following \cite{Bretagnolle.Dacunha-Castelle.Krivine.1965}.  

\begin{theorem}\label{thm:BDCK}
Suppose $1 \leq p \leq r \leq 2$.  Then, $L^r([0,1]; \R)$ isometrically embeds into $L^p([0,1]; \R)$.
\end{theorem}

Their proof first skillfully employs a beautiful application of probability theory, namely the theory of stable random variables, to build an explicit embedding of $\ell^r(\R)$ into $L^p([0,1]; \R)$.  It then uses an unexpected application of ultraproducts to obtain an embedding of $L^r([0,1]; \R)$ into $L^p([0,1]; \R)$.  According to W. Henson, this was the first application of ultraproducts in functional analysis \cite{Henson.Iovino.2002}.  A more constructive but more involved proof via Poisson processes later appeared in the monograph by Lindenstrauss and Tzafriri \cite{Lindenstrauss.Tzafriri.1979}.   Both proofs rely on the theory of abstract real $L^p$ spaces.

The complex version of Theorem \ref{thm:BDCK} was proven by Herz \cite{Herz.1971} and later by Rosenthal \cite{Rosenthal.1973}.

Ultraproducts originated in mathematical logic.  Here, we will give another proof of Theorem \ref{thm:BDCK} via another tool from logic, namely the compactness theorem of continuous logic.  Continuous logic is, roughly speaking, model theory for continuous structures such as Banach spaces.  We will also prove the complex version of Theorem \ref{thm:BDCK}.  Thus, our main theorem is the following.

\begin{theorem}\label{thm:main}
Suppose $1 \leq p \leq r \leq 2$.  Then, $L^r([0,1]; \F)$ isometrically embeds into $L^p([0,1]; \F)$.
\end{theorem}

As we will discuss in Section \ref{sec:conclusion}, the resulting proof is constructive and avoids the use of ultraproducts or other forms of the axiom of choice.  We also believe it is more direct than the proof in \cite{Lindenstrauss.Tzafriri.1979}.  

We aim for accessibility to the communities of functional analysis and mathematical logic.  Accordingly, the paper is organized as follows.  In Section \ref{sec:bak.prelim.fa} we first cover the minimal background required from functional analysis beyond what is normally covered in standard graduate courses.  We then introduce the concepts of \emph{$L^p$-formally disjointly supported} vectors and \emph{$L^p$-formal disinintegrations} which underpin our use of the compactness theorem from continuous logic.  We also set forth what appears to be a new definition of abstract complex $L^p$ spaces and show 
that this notion indeed characterizes complex $L^p$ spaces.  In Section \ref{sec:bak.prelim.CL}, we begin by summarizing the essentials of continuous logic and then review the representation of Banach lattices and abstract real $L^p$-spaces in this framework.  We then show that our class of abstract complex $L^p$ spaces can be represented in continuous logic as well.  

Our proof of Theorem \ref{thm:main} is then given in Section \ref{sec:proof}.  A few concluding remarks are given in Section \ref{sec:conclusion}.  In an appendix, we give a proof of a relevant non-supporting result in Banach lattice theory which is well-known although its proof does not seem to be explicitly recorded in the literature. 

\section{Background and preliminaries from functional analysis}\label{sec:bak.prelim.fa}

\subsection{Background from functional analysis}\label{sec:bak.prelim.fa::subsec:bak}

If $\mathcal{B}$ is a Banach space, and if $X \subseteq \mathcal{B}$, we write $\langle X \rangle$ for the closed subspace of $\mathcal{B}$ generated by $X$.

We first cover required inclusion and embedding results.  We then deal with Banach lattices and abstract real $L^p$ spaces.  

\subsubsection{Inclusion and embedding results}

Recall that a measure space $\Omega = (X, \mathcal{M}, \mu)$ is \emph{separable} if there is a countable $\mathcal{D} \subseteq \mathcal{M}$ so that whenever $\mu(A) < \infty$ and $\epsilon >0$ there is a $B \in \mathcal{D}$ so that $\mu(A \triangle B) < \epsilon$.  The following lemma is fairly well-known.  

\begin{lemma}\label{lm:sep.sub}
If $X$ is a separable subspace of $L^p(\Omega; \F)$, then there is a separable subspace 
$\Omega_0$ of $\Omega$ so that $X \subseteq L^p(\Omega_0; \F)$. 
\end{lemma}

The next theorem is a consequence of the classification of separable $L^p$-spaces.

\begin{theorem}\label{thm:sep.embed}
Every separable $L^p(\F)$ space isometrically embeds into $L^p([0,1]; \F)$.
\end{theorem} 

The proof of Theorem \ref{thm:main} hinges on the following.

\begin{theorem}\label{thm:lr.Lp}
If $1 \leq p \leq r \leq 2$, then there is an isometric embedding of $\ell^r(\F)$ into $L^p([0,1]; \F)$.  
\end{theorem}

As noted in the introduction, the real case of Theorem \ref{thm:lr.Lp} was first proven by Bretagnolle, Dachuna-Castelle, and Krivine \cite{Bretagnolle.Dacunha-Castelle.Krivine.1965}.  It will be useful for us to review the key elements of their proof.  We begin with some probability theory our sources for which are Chapter 3 of \cite{Durrett.2010}, Chapter 2 of \cite{Ushakov.1999}, and Chapters 1 and 2 (especially Sections 2.5 and 2.6) of \cite{Samorodnitsky.Taqqu.1994}.

Suppose $g$ is a random variable.  The \emph{characteristic function} $\Phi_g : \R \rightarrow \R$ of $g$ is defined by $\Phi_g(t) = E[e^{igt}]$.  
It is well-known that $\Phi_g = \Phi_h$ only when $g$ and $h$ have the same distribution. 

A random variable $g$ is \emph{$r$-stable} if whenever $g_1$, $g_2$ are independent random variables with the same distribution as $g$ and $a,b > 0$,
$ag_1 + bg_2$ has the same distribution as  $\sqrt[r]{a^r + b^r}g + c$ for some real number $c$.  It is well-known that $r$-stable random variables exist only when $0 < r \leq 2$.  The case $r=2$ corresponds to Gaussian random variables.  
The $r$-stable random variables can be characterized by the forms of their characteristic functions.  In particular, $t \mapsto \exp(-\sigma^r |t|^r)$ is the characteristic function of an $r$-stable random variable whenever $0 < r \leq 2$ and $\sigma \geq 0$.  Furthermore, any random variable with such a characteristic function is $r$-stable.  These are known as the \emph{symmetric $r$-stable} random variables.    Suppose $g$ is symmetric $r$-stable with $\sigma > 0$ (i.e. $\Phi_g$ is non-constant).  Then, if $g_0$ and $g_1$ are independent random variables with the same distribution as $g$, and if $|a|^r + |b|^r = 1$, it follows that $ag_0 + bg_1$ has the same distribution
as $g$.  

Now, suppose $g$ is a symmetric $r$-stable random variable on the unit interval with $\sigma > 0$ and $1 \leq r \leq 2$.  If $r < 2$, then $P[|g(x)| > t]$ is asymptotic to $t^{-r}$ and so $g$ has finite $L^p$ norm when $1 \leq p < r$.   If $p = 2$, then $g$ has finite $L^q$ norm for all $q \geq 1$ since all of its absolute moments are finite.   By a result of Kolomogorov, there is an infinite independent family $\{g_n\}_{n \in \N}$ of random variables on $[0,1]$ each of which has the same distribution as $g$.  It then follows that $e_n \mapsto \norm{g_n}_p^{-1} g_n$ is an isometric embedding of $\ell^r(\R)$ into $L^p([0,1]; \R)$.  

We now discuss complex random variables (i.e. $2$-dimensional random vectors).  For these random variables, $r$-stability is defined as in the real case except the constant $d$ is complex.  If a complex random variable is $r$-stable, then its real and imaginary parts are 
$r$-stable.  The characteristic function $\Phi_g$ of a complex random variable $g$ is defined by $\Phi_g(z) = E[\exp(i\Re(\overline{z}g)]$ (i.e. scalar multiplication is replaced by the inner product).  For each $c \geq 0$ there is an $r$-stable complex random variable $g$ so that $\Phi_g(z) = \exp(-c|z|^r)$ and any complex random vector with such a characteristic function is $r$-stable.  Once again, these are called symmetric.  The proof of Theorem \ref{thm:lr.Lp} in the complex case can now be effected in the same manner as the real case.

\subsubsection{Banach lattices and abstract $L^p$ spaces}

Our source for this subsection is \cite{Meyer-Nieberg.1991}.

A \emph{Banach lattice} consists of a real Banach space $\B$ together with a lattice ordering $\leq$ of $\B$ with the following properties.
\begin{itemize}
	\item If $v_0,v_1,u \in \B$, and if $v_0 \leq v_1$, then $v_0 + u \leq v_1 + u$.  
	
	\item If $u,v \in \B$ are such that $u \leq v$, and if $a$ is a nonnegative real, then 
	$au \leq av$.
	
	\item $\norm{u}_\B \leq \norm{v}_\B$ whenever $u,v \in \mathcal{B}$ and $|u| \leq |v|$ (where $|a| = a \join (-a)$).
\end{itemize}   
If $u$ is a vector of a Banach lattice we write $u^+$ for $u \join \zerovec$, $u^-$ for $(-u) \join \zerovec$.   Furthermore, we say that $u,v \in \mathcal{B}$ are \emph{disjoint} if $|u| \meet |v| = \zerovec$.

When $f,g \in L^p(\Omega; \R)$, let $f \leq g$ hold if and only if 
$f(t) \leq g(t)$ almost everywhere.  It follows that $(L^p(\Omega; \R), \leq)$ is a Banach lattice.

When $1 \leq p < \infty$, a Banach lattice is an \emph{abstract real $L^p$-space} if it satisfies the condition that 
$\norm{x + y}^p = \norm{x}^p + \norm{y}^p$ whenever $|x| \meet |y|  = \zerovec$.   Clearly, every real $L^p$ space is an abstract real $L^p$-space.  Moreover, the converse is true as was proven by Kakutani in 1941 \cite{Kakutani.1941} (see also Nakano \cite{Nakano.1941}).

\begin{theorem}[Kakutani Representation Theorem]\label{thm:kakutani}
Every abstract real $L^p$ space is isometrically isomorphic to a real $L^p$-space.  
\end{theorem}

The following characterization of abstract real $L^p$-spaces will be very useful when representing them in continuous logic.

\begin{theorem}\label{thm:abstr.Lp.equiv}
Let $(\mathcal{B}, \leq)$ be a Banach lattice, and suppose $1 \leq p < \infty$.  
Then, $(\mathcal{B}, \leq)$ is an abstract real $L^p$ space if and only if 
$\norm{u + v}_\B^p \geq \norm{u}_\B^p + \norm{v}_\B^p$ whenever $u,v \geq \zerovec$.  
\end{theorem}

The proof of Theorem \ref{thm:abstr.Lp.equiv} is highly non-trivial; see Chapter 17 of \cite{Zaanen.1983}.

We will consider complex abstract $L^p$ spaces in the following subsection.

\subsection{Preliminaries from functional analysis}\label{sec:bak.prelim.fa::subsec:prelim}

\subsubsection{Disintegrations}

Suppose $1 \leq p < \infty$, $\mathcal{B}$ is a Banach space, and $v_1, \ldots, v_n \in \mathcal{B}$. We say $v_1, \ldots, v_n$ are \emph{$L^p$-formally disjointly supported} if 
	\[
	\norm{\sum_j \alpha_j v_j}_\B^p = \sum_j |\alpha_j|^p \norm{v_j}_\B^p
	\]
	for all scalars $\alpha_1, \ldots, \alpha_n$.
If $f_1, \ldots, f_n \in L^p(\Omega;\F)$ are disjointly supported, then they are $L^p$-formally disjointly supported. 
By a result of J. Lamperti, if $p \neq 2$, then $L^p$-formally disjointly supported vectors in $L^p(\Omega; \F)$ are disjointly supported \cite{Lamperti.1958}.  If $X,Y \subseteq \mathcal{B}$ have the property that $u,v$ are $L^p$-formally disjointly supported whenever $u \in X$ and $v \in Y$, then we say that $X,Y$ are $L^p$-formally disjointly supported.

Let $\N$ denote the set of nonnegative integers.  Let $\N^*$ denote the set of all finite sequences of nonnegative integers including the empty sequence.  When $\sigma \in \N^*$, $|\sigma|$ denotes the length of $\sigma$.  If $\sigma, \tau \in \N^*$, we write $\sigma \sqsubset \tau$ if $\sigma$ is a prefix of $\tau$.  If $\sigma \sqsubset \tau$, and if $|\tau| = |\sigma| + 1$, then 
we say that $\tau$ is a \emph{child} of $\sigma$.  By a \emph{tree} we mean a nonempty subset of $\N^*$ that is closed under prefixes.  We write $\sigma\cat\tau$ for the concatenation of $\sigma$ with $\tau$.

Suppose $\mathcal{B}$ is a Banach space and $\phi : S \rightarrow \mathcal{B}$ where $S \subseteq \N^*$ is a tree.  $\phi$ is \emph{summative} if for every $\nu \in S$, $\phi(\nu) = \sum_{\nu'} \phi(\nu')$ where $\nu'$ ranges over all the children of $\phi$ in $S$.  $\phi$ is \emph{formally $L^p$-separating} if 
$\phi(\nu_1), \ldots, \phi(\nu_n)$ are formally $L^p$-disjointly supported whenever $\nu_1, \ldots, \nu_n \in S$ are incomparable.  Finally, $\phi$ is an \emph{$L^p$-formal disintegration} of $\mathcal{B}$ if 
it is summative, formally $L^p$-separating, never zero, and its range is linearly dense (i.e. $\B$ is the closure of the linear span of $\ran(\phi)$.

Suppose $\mathcal{B}_0$ and $\mathcal{B}_1$ are Banach spaces and $\phi_j : S_j \rightarrow \mathcal{B}_j$ 
for each $j \in \{0,1\}$.  A map $f : S_0 \rightarrow S_1$ is an \emph{isomorphism} of $\phi_0$ and $\phi_1$ if 
if is an order isomorphism (with respect to $\sqsubseteq$) of $S_0$ onto $S_1$ and if $\norm{\phi_1(f(\nu))}_{\B_1} = \norm{\phi_0(\nu)}_{\B_0}$ for all $\nu \in S_0$.  

Our main result on $L^p$-formal disintegrations (Theorem \ref{thm:lifting} below) is that isomorphisms of $L^p$-formal disintegrations induce isometric isomorphisms of the corresponding spaces.  To this end, we first prove the following.

\begin{lemma}\label{lm:sum}
Suppose $S \subseteq \N^*$ is a tree and $\phi : S \rightarrow \mathcal{B}$ is summative.  
Suppose $F$ is a finite subtree of $S$, and let $\{\beta_\nu\}_{\nu \in F}$ be a family of scalars.  Let $F'$ denote the terminal nodes of $F$.  Then, 
\[
\sum_{\nu \in F} \beta_\nu \phi(\nu) = \sum_{\nu \in F'} \left(\sum_{\mu \sqsubseteq \nu} \beta_\mu \right) \phi(\nu).
\]
\end{lemma}

\begin{proof}
Let $n$ denote the height of $F$.  
We proceed by induction on $n$.  There is nothing to prove if $n = 0$, so suppose $n > 0$ and that the claim holds for $n - 1$.  For each $\nu \in F$, let $\gamma_\nu = \sum_{\mu \sqsubseteq \nu} \beta_\mu$.  
Let $F''$ denote the set of $\sqsubseteq$-maximal nodes of $F - F'$.  
By the induction hypothesis, we have, 
\[
\sum_{\nu \in F} \beta_\nu \phi(\nu) = \sum_{\nu \in F''} \gamma_\nu \phi(\nu) + \sum_{\nu \in F'} \beta_\nu \phi(\nu).
\]
For each $\nu \in F''$, let $F'_\nu$ denote the set of all children of $\nu$ that belong to $F'$.
Since $\phi$ is summative, for each $\nu \in F''$, 
\[
\gamma_\nu \phi(\nu) + \sum_{\mu \in F''_\nu} \beta_\mu \phi(\mu) = 
\sum_{\mu \in F''_\nu} (\gamma_\nu + \beta_\mu) \phi(\mu) = \sum_{\mu \in F''_\nu} \gamma_\mu \beta_\mu.
\]
Therefore,
\begin{eqnarray*}
\sum_{\nu \in F''} \gamma_\nu \phi(\nu) + \sum_{\nu \in F'} \beta_\nu \phi(\nu) & = & 
\sum_{\nu \in F''} \gamma_\nu \phi(\nu) + \sum_{\nu \in F''} \sum_{\mu \in F''_\nu} \beta_\mu \phi(\mu)\\
& = & \sum_{\nu \in F''} \sum_{\mu \in F''_\nu} \gamma_\mu \phi(\mu) \\
& = & \sum_{\nu \in F'} \gamma_\nu \phi(\nu).
\end{eqnarray*}
\end{proof}

\begin{theorem}\label{thm:lifting}
Suppose $\phi_j$ is a $L^p$-formal disintegration of $\mathcal{B}_j$ for each $j \in \{0,1\}$, and suppose 
$f$ is an isomorphism of $\phi_0$ with $\phi_1$.  Then, there is a unique isometric isomorphism 
$T_f$ of $\mathcal{B}_0$ onto $\mathcal{B}_1$ so that $T_f(\phi_0(\nu)) = \phi_1(f(\nu))$ for all 
$\nu \in \dom(\phi_0)$.  
\end{theorem}

\begin{proof}
Let $S_j = \dom(\phi_j)$, and let $X_j$ denote the linear span of $\ran(\phi_j)$.  When $F \subseteq S_0$ is finite and $\{\alpha_\nu\}_{\nu \in F}$ is a family of scalars, let 
\[
T(\sum_{\nu \in F} \alpha_\nu \phi_0(\nu)) = \sum_{\nu \in F} \alpha_\nu \phi_1(f(\nu)).
\]

We first show that $T$ is a well-defined map on $X_0$.  Let $F \subseteq S_0$ be finite.  It suffices to show that 
$\sum_{\nu \in F} \alpha_\nu \phi_1(f(\nu)) = \zerovec$ whenever $\{\alpha_\nu\}_{\nu \in F}$ is a family of scalars so that $\sum_{\nu \in F} \alpha_\nu \phi_0(\nu) = \zerovec$.  So, suppose $\sum_{\nu \in F} \alpha_\nu \phi_0(\nu) = \zerovec$.  Without loss of generality, we may assume $F$ is a tree.  Let $F'$ denote the set of all terminal nodes of $F$, and let $\gamma_\nu = \sum_{\mu \sqsubseteq \nu} \beta_\mu$.  By Lemma \ref{lm:sum}, 
\[
\zerovec = \sum_{\nu \in F'} \gamma_\nu \phi_0(\nu).
\]
Since $\phi_0$ is formally $L^p$-separating and never zero, it follows that $\gamma_\nu = 0$ for each 
$\nu \in F'$.  Since $f$ is an isomorphism, $\phi_1 \circ f$ is an $L^p$-formal disintegration.  
Thus, by Lemma \ref{lm:sum} again, 
\[
\sum_{\nu \in F} \alpha_\nu \phi_1(f(\nu)) = \sum_{\nu \in F'} \gamma_\nu \phi_1(f(\nu))  \zerovec.
\] 
Thus, $T$ is well-defined.

By definition, $T$ is linear.  Since $f$ is an isomorphism, $\ran(T) = X_1$.  

We now show $T$ is isometric.  Suppose $F \subseteq S_0$ is finite and $\{\alpha_\nu\}_{\nu \in F}$ is a family of scalars.  Again, let $n = \max\{ |\nu|\ :\ \nu \in F\}$, and let $\beta_\nu$ be defined as above.  
Then, by Lemma \ref{lm:sum},
\begin{eqnarray*}
\norm{\sum_{\nu\in F} \alpha_\nu \phi_0(\nu)}_{\B_0}^p & = & \norm{\sum_{\nu \in S_0 \cap \N^n} \left( \sum_{\mu \sqsubseteq \nu} \beta_\mu\right) \phi_0(\nu)}_{\B_0}^p
\end{eqnarray*}
Since $\phi_0$ is $L^p$-formally separating, 
\begin{eqnarray*}
\norm{\sum_{\nu \in S_0 \cap \N^n} \left( \sum_{\mu \sqsubseteq \nu} \beta_\mu\right) \phi_0(\nu)}_{\B_0}^p & = & \sum_{\nu \in S_0 \cap \N^n} \left| \sum_{\mu \sqsubseteq \nu} \beta_\mu\right|^p \norm{\phi_0(\nu)}_{\B_0}^p
\end{eqnarray*}
Since $f$ is an isomorphism, 
\begin{eqnarray*}
\sum_{\nu \in S_0 \cap \N^n} \left| \sum_{\mu \sqsubseteq \nu} \beta_\mu\right|^p \norm{\phi_0(\nu)}_{\B_0}^p & = & \sum_{\nu \in S_0 \cap \N^n} \left| \sum_{\mu \sqsubseteq \nu} \beta_\mu\right|^p \norm{\phi_1(f(\nu))}_{\B_1}^p
\end{eqnarray*}
Since $\phi_1$ is $L^p$-formally separating, 
\begin{eqnarray*}
\sum_{\nu \in S_0 \cap \N^n} \left| \sum_{\mu \sqsubseteq \nu} \beta_\mu\right|^p \norm{\phi_1(f(\nu))}_{\B_1}^p & = & \norm{\sum_{\nu \in S \cap \N^n} \left(\sum_{\mu \sqsubseteq \nu} \beta_\mu\right) \phi_1(f(\nu))}_{\B_1}^p
\end{eqnarray*}
Again by Lemma \ref{lm:sum} and the definition of $T$, 
\begin{eqnarray*}
\norm{\sum_{\nu \in S \cap \N^n} \left(\sum_{\mu \sqsubseteq \nu} \beta_\mu\right) \phi_1(f(\nu))}_{\B_1}^p & = & \norm{T(\sum_{\nu\in F} \alpha_\nu \phi_0(\nu))}_{\B_1}^p
\end{eqnarray*}

Thus, $T$ extends to an isometric isomorphism of $\mathcal{B}_0$ onto $\mathcal{B}_1$.  
The uniqueness of $T$ follows from the linear density of the range of $\phi_0$.
\end{proof}

We note that Theorem \ref{thm:lifting} is an extension of Theorem 4.2 of \cite{Clanin.McNicholl.Stull.2018}.

\subsubsection{Abstract complex $L^p$ spaces}

Suppose $V$ is a real vector space.  The \emph{complexification of $V$} is the complex vector 
space over $V \times V$ in which addition is defined coordinatewise and scalar multiplication is defined by 
\[
(x + iy) (v_0, v_1) = (xv_0 - yv_1, yv_0 + xv_1).
\]

We denote the complexification of $V$ by $V_\C$.  If $\bfv = (v_0, v_1) \in V_\C$, let $\Re(\bfv) = v_0$ and $\Im(\bfv) = v_1$; in addition we denote $\bfv$ by $v_0 + i v_1$.  Let $\Re(V_\C) = V \times \{0\}$, and let $\Im(V_\C) = \{0\} \times V$.

Under certain conditions, it is possible to construct a norm on $\B_\C$ by first defining a modulus by $|\bfv| = \sup_\theta \Re(e^{i\theta}\bfv)$ (where the supremum is taken with respect to the ordering on $\B$) and then setting 
$\norm{\bfv} = \norm{|\bfv|}$.  As will be discussed in the appendix, the complexification of $L^p(\Omega; \R)$ is $L^p(\Omega; \C)$.  However, this construction is not amenable to continuous logic.  
We go around this obstacle by adding a norm on $\B_\C$ and a condition on this norm that ensures it will behave as an $L^p$ norm.   This condition is easily representable in continuous logic.

\begin{definition}\label{def:complex.abstract.Lp}
Suppose $1 \leq p < \infty$, and let $(\mathcal{B}, \leq)$ be an abstract real $L^p$-space.  
Let $\norm{\ }$ be a norm on $\mathcal{B}_\C$ so that the following hold.
\begin{enumerate}
	\item For all $v \in \mathcal{B}$, $\norm{v}_\mathcal{B} = \norm{v + i\zerovec}$.  
	
	\item If $v_0, v_1 \in \mathcal{B}$ are disjoint, then  
	$v_0 + i\zerovec$ and $v_1 + i\zerovec$ are formally $L^p$-disjointly supported (with respect to $\norm{\ }$).  i.e. 
	\[
	\norm{\alpha(v_0 + i\zerovec) + \beta(v_1 + i\zerovec)}^p = 
	|\alpha|^p \norm{v_0}_\B^p + |\beta|^p \norm{v_1}_B^p
	\]
	for all \emph{complex} scalars $\alpha, \beta$.
\end{enumerate}
We call $(\mathcal{B}_\C, \norm{\ })$ an \emph{abstract complex $L^p$ space}.
\end{definition}
If $(\mathcal{B}_\C, \norm{\ })$ is an abstract complex $L^p$ space then we also have 
$\norm{v}_\B = \norm{\zerovec + iv}$.  

\begin{theorem}\label{thm:complex.abstract.Lp}
If $(\mathcal{B}_\C, \norm{\ })$ is an abstract complex $L^p$ space, then there is a measure space $\Omega$ so that $(\mathcal{B}_\C, \norm{\ })$ is isometrically isomorphic to $L^p(\Omega; \C)$.
\end{theorem}

\begin{proof}
By Theorem \ref{thm:kakutani}, there is a measure space $\Omega$ so that there is an isometric isomorphism $T$ of $(L^p([0,1]; \R),\leq)$ onto $(\mathcal{B}, \leq)$.  Let $\Omega = (X, \mathcal{M}, \mu)$.  

For each complex simple function $s = \sum_{j \leq n} \alpha_j \Ind_{A_j}$ of $\Omega$, let 
$T_1(s) = \sum_{j \leq n} \alpha_j(T(\Ind_{A_j}), \zerovec)$.  
Since $\Ind_{A \cup B} = \Ind_A + \Ind_B$ when $A,B$ are disjoint, it follows that $T_1$ is 
well-defined.  It follows from the definition of $\mathcal{B}_\C$ that $T_1$ is linear.  The conditions of Definition \ref{def:complex.abstract.Lp} ensure that $T_1$ is an isometry.  
Thus $T_1$ has a unique extension to $L^p(\Omega; \C)$ which we denote by $T_1$ as well.  
It follows that $T_1$ is a linear isometric map of $L^p(\Omega;\C)$ into $\B_\C$.  
$T_1$ extends $T$ in the sense that if $f \in L^p(\Omega; \R)$, then 
$T_1(f) = (T(f), \zerovec)$.  

We claim that $T_1$ is surjective.  For, let $v = (v_0, v_1) \in \B_\C$.  Let 
$\epsilon > 0$.  There is a real simple function $s_k = \sum_{j < n_k} t_{k,j} \Ind_{A_{k,j}}$ 
of $\Omega$ so that $\norm{T^{-1}(v_k) - s_k}_p < \epsilon /2$.  
It follows that $\norm{(v_0,v_1) - T_1(s_0 + i s_1)} < \epsilon$.  Thus, the range of $T_1$ is dense in $L^p(\Omega; \C)$ and so $T_1$ is surjective.
\end{proof}

\section{Background and preliminaries from continuous logic}\label{sec:bak.prelim.CL}

\subsection{Background from continuous logic}\label{sec:bak.prelim.CL::subsec:bak}

We refer to sections 1 - 5 of \cite{Ben-Yaacov.Berenstein.Henson.Usvyatsov.2008} for a thorough treatment of the rudiments of continuous logic which we summarize here.
  
First-order logic was developed primarily to study algebraic and combinatorial structures.  Continuous logic is an extension of first-order logic to metric structures.  Unlike first-order logic, in continuous logic truth values range between $0$ and $1$ inclusive.  The set of connectives is larger: every continuous function $u: [0,1]^n \rightarrow [0,1]$ is regarded as a connective.  The quantifiers are `$\inf$' and `$\sup$'.  However, the equality sign is replaced by a distance symbol $\dsym$.  

A \emph{metric language} consists of relation symbols, function symbols, and constant symbols.  In addition, with each function or relation symbol $\phi$ there is associated a \emph{modulus} $\Delta_\phi : (0, \infty) \rightarrow (0,\infty)$.  Terms are built from constants, variables, and function symbols according to the usual rules.  Well-formed formulas are built from terms, $\dsym$, relation symbols, connectives, and quantifiers according to the usual rules.  A variable of a wff $\Phi$ is \emph{free} if it is not governed by any quantifier of $\Phi$.  
A \emph{sentence} is a well-formed formula (wff) with no free variables.  A \emph{theory} is a set of sentences.

With some modifications, \emph{interpretations} of metric languages are defined as in classical first-order logic. Intuitively, an interpretation of a metric language is a way of assigning meaning to its constituents.  Unlike first-order logic, the domain of an interpretation of a metric language is a complete metric space $(M,d)$ of diameter $1$.  The distance symbol $\dsym$ is interpreted by the metric $d$.  The interpretation of an $n$-ary function symbol is a uniformly continuous function from $M^n$ into $M$.  The interpretation of an $n$-ary relation symbol is a uniformly continuous map from $M^n$ into $[0,1]$.  The moduli functions must serve as moduli of continuity for interpretations of function and relation symbols.  Constants are still interpreted by constants.  Thus, each interpretation $\mathcal{A}$ of a metric language $\mathcal{L}$ assigns to each sentence $\Phi$ of $\mathcal{L}$ a truth value in $[0,1]$ which we denote by $\Phi^\mathcal{A}$.  

In general, a relation symbol will represent a distance function for some closed set; i.e. a function of the form $d(p, X) = \inf\{q \in X\ :\ d(p,q)\}$.  Thus, $0$ is regarded as representing `true' instead of $1$.  
Accordingly, we say that an interpretation $\mathcal{A}$ of a metric language $\mathcal{L}$
\emph{satisfies} a sentence $\Phi$ of $\mathcal{L}$ if $\Phi^\mathcal{A} = 0$.  An interpretation satisfies a theory if it satisfies each sentence in the theory in which case it is said to be a \emph{model} of the theory.

The compactness theorem holds in continuous logic: a theory has a model if all of its finite subsets do.

\begin{definition}\label{def:represent}
Let $\mathcal{K}$ be a class of structures.  We say that a theory $T$ of a metric language 
$\mathcal{L}$ \emph{represents} $\mathcal{K}$ if the following hold.
\begin{enumerate}
	\item For every $\B \in \mathcal{K}$, there is an interpretation $\mathcal{A}_\B$ of $\mathcal{L}$ so that $\mathcal{A}_\B \models T$.  
	
	\item If $\B_0, \B_1 \in \mathcal{K}$, then $\B_0$ is isomorphic to $\B_1$ if and only if $\mathcal{A}_{\B_0}$ is isomorphic to 
	$\mathcal{A}_{\B_1}$.
	
	\item If $\mathcal{A} \models T$, then there is a $\B_\mathcal{A} \in \mathcal{K}$ so 
	that $\mathcal{A}$ isomorphically embeds in $\mathcal{A}_{\B_\mathcal{A}}$.  
\end{enumerate}
\end{definition}

We now discuss the representation of Banach lattices in continuous logic.  We take the approach of representing the behavior of the vector space and lattice operations on the unit ball.  

$\LB$ is the metric language that consists of the following.
\begin{enumerate}
	\item For all scalars $s,t$ so that $|s| + |t| \leq 1$, a unique binary function symbol $\plus_{s,t}$. 
	
	\item A unary relation symbol $\norm{\ }$.
	
	\item A constant symbol $\zerovec$.
	
	\item $\Delta_{\plus_{s,t}}(\epsilon) = 2\epsilon$.  
	
	\item $\Delta_{\norm{\ }}(\epsilon) = \epsilon$.  
\end{enumerate}

Of course, when working with $\LB$, we write  $s \tau_1 + t \tau_2$ for $\plus_{s,t}(\tau_1, \tau_2)$ and $s \tau_1$ for $\plus_{s,1}(\tau_1, \zerovec)$.

$\LBL$ is the metric language that  consists of $\LB$ together with a family 
$\{\meet_{s,t}\}_{|s| + |t| \leq 1}$ of new and distinct binary operation symbols and a new 
unary operation symbol $|\ |$.

When working with $\LBL$ we write $(s \tau_1 \meet t \tau_2)$ for $\meet_{s,t}(\tau_1, \tau_2)$, $(s \tau_1 \join t \tau_2)$ for $-((-s)\tau_1 \meet (-t) \tau_2)$, and $\tau_1^+$ for $(1\tau_1) \vee (0\tau_1)$.

The following essentially follows from \cite{Ben-Yaacov.Berenstein.Henson.Usvyatsov.2008} and also \cite{Ben-Yaacov.2009}.

\begin{theorem}\label{thm:rep.Banach.lattice}
There is a theory $\TBL$ of $\LBL$ that represents the class of Banach lattices.
\end{theorem}

Let $\leq$ denote the connective $\leq(s,t) = |t - \max\{s,t\}|$.  We generally use infix rather then prefix notation with this connective.

$\TLpR$ is the theory of $\LBL$ that consists of the conditions of $\TBL$ together with 
\[
\sup_{x_0,x_1} \norm{sx_0^+}^p + \norm{tx_1^+}^p \leq \norm{sx_0^+ + tx_1^+}^p
\]
whenever $0 \leq s \leq 1$ and $0 \leq t \leq 1 - s$.  

By Theorem \ref{thm:abstr.Lp.equiv}, we obtain the following.

\begin{theorem}\label{thm:ext.Lp}
$\TLpR$ represents the class of abstract real $L^p$ spaces.  
\end{theorem}

\subsection{Preliminaries from continuous logic}\label{sec:bak.prelim.CL::subsec:prelim}

Let $\LLpC$ consist of $\LBL$ together with a new binary predicate symbol 
$\norm{\ }_+$.  Our goal in this subsection is to prove the following.

\begin{theorem}\label{thm:complex.Lp.rep}
There is a theory $\TLpC$ of $\LLpC$ that represents the class of abstract complex $L^p$ spaces.  
\end{theorem}

We divide the proof into the following two lemmas.

\begin{lemma}\label{lm:ds.comp}
If $f_0, f_1 \in L^p(\Omega; \F)$, then there exist disjointly supported $g_0, g_1 \in L^p(\Omega; \F)$ 
so that 
\[
\max\{\norm{g_0 - f_0}_p, \norm{g_1 - f_1}_p\} \leq \norm{|f_0| \meet |f_1|}_p.
\]
\end{lemma}

\begin{proof}
Let $\Omega = (X, \mathcal{M}, \mu)$.  For each $t \in X$, let 
\[
g_j(t) = \left\{ \begin{array}{cc}
		0 & \mbox{if $|f_j(t)| \leq |f_{1 - j}(t)|$}\\
		f_j(t) & \mbox{otherwise}\\
		\end{array}
		\right.
\]
It follows that $g_0$ and $g_1$ are disjointly supported.  If $t \in X$, then 
$|f_j (t)- g_j(t)|^p \leq \min\{|f_0(t)|^p, |f_1(t)|^p\}$.  Thus, $\norm{g_j - f_j}_p \leq \norm{|f_0| \meet |f_1|}_p$.  
\end{proof}

\begin{lemma}\label{lm:complex.cond}
Suppose that $(\B, \leq)$ is a Banach lattice and that $\norm{\ }$ is a norm on $\mathcal{B}_\C$.  Then, $(\B_\C, \norm{\ })$ is a complex abstract $L^p$ space if and only if for each $v_0, v_1 \in \B$ and $\alpha, \beta \in \C$, 
\begin{equation}
\inf_{u_0, u_1} \max\{
G_0(u_0, u_1,v_0,v_1), G_{1, \alpha, \beta}(u_0,u_1,v_0,v_1), G_2(v_0)\} = 0 \label{eqn:complex.cond}
\end{equation}
where:
\begin{eqnarray*}
G_0(u_0,u_1,v_0,v_1) & = & \leq(\max\{\norm{v_0 - u_0}_\B, \norm{v_1 - u_1}_\B\}, \norm{|v_0| \wedge |v_1|})\\ 
G_{1, \alpha, \beta}(u_0,u_1,v_0,v_1) & = & |\norm{\alpha (u_0+i\zerovec) + \beta (u_1+i\zerovec)}^p - |\alpha|^p \norm{u_0}_\B^p - |\beta|^p \norm{u_1}_\B^p|\\
G_2(v_0) & =& |\norm{v_0 + i\zerovec} - \norm{v_0}_\B|
\end{eqnarray*}
\end{lemma}

\begin{proof}
If $(\B_\C, \norm{\ })$ is a complex abstract $L^p$ space, then (\ref{eqn:complex.cond}) follows from 
Theorem \ref{thm:complex.abstract.Lp} and Lemma \ref{lm:ds.comp}.  Suppose (\ref{eqn:complex.cond}) holds for all $v_0, v_1 \in \B$ and $\alpha, \beta \in \C$.  
We immediately obtain that $\norm{v + i\zerovec} = \norm{v}_\B$.  
Let $v_0, v_1 \in \B$ be disjoint.  
Thus, $G_0(u_0, u_1, v_0, v_1) = \max\{\norm{v_0 - u_0}_\B, \norm{v_1 - u_1}_\B\}$.  
Hence, for each 
$n \in \N$, there exist $u_{0,n}, u_{1,n} \in \B$ so that 
$\norm{v_j - u_{j,n}} < 2^{-n}$ and so that 
\[
|\norm{\alpha (u_{0,n}+i\zerovec) + \beta (u_{1,n}+i\zerovec)}^p - |\alpha|^p \norm{u_{0,n}}_\B^p - |\beta|^p \norm{u_{1,n}}_\B^p| < 2^{-n}.
\]
Thus, on the one hand, $v_j = \lim_n u_{j,n}$.  
We can then infer that, $v_0 + i\zerovec$, $v_1 + i \zerovec$ are formally $L^p$-disjointly supported.
\end{proof}

It is now fairly straightforward to formulate, for each $\alpha, \beta \in \C$ with $|\alpha| + |\beta| \leq 1$, a sentence of $\LLpC$ that asserts (\ref{eqn:complex.cond}) holds for all $v_0, v_1$.  Theorem \ref{thm:complex.Lp.rep} now follows.

\section{Proof of Theorem \ref{thm:main}}\label{sec:proof}

We first consider the real case.  
Suppose $1 \leq p \leq r \leq 2$.  Let $\mathcal{L}^+$ be the language formed by adjoining a countable 
family $\{c_\sigma\}_{\sigma \in \{0,1\}^*}$ of new and distinct constant symbols to $\LBL$.  

We define sentences of $\mathcal{L}^+$ as follows.  Let $\sigma \in \{0,1\}^*$.   When $|s| \leq 1/2$, let 
\[
\Phi_{\sigma,s}  =  \dsym(sc_\sigma, sc_{\sigma\cat(0)} + sc_{\sigma\cat(1)}).
\]
When $|s| + |t| \leq 1$, let 
\[
\Psi_{\sigma,s,t}  =  | \norm{s c_{\sigma\cat(0)} + t c_{\sigma\cat(1)}}^r - |s|^r\norm{c_{\sigma\cat(0)}}^r - |t|^r\norm{c_{\sigma\cat(1)}}^r|
\]
Finally, let $\Gamma_\sigma  = |\norm{c_\sigma} - 2^{-|\sigma|/r}|$.  Informally speaking, the sentences 
$\Phi_{\sigma,s}$ and $\Phi_{\sigma,s,t}$ altogether say that the map $\sigma \mapsto c_\sigma$ is an $L^r$-formal disintegration.

Let 
\[
T_n = T_{L^p}\ \cup 
\{\Phi_{\sigma,s}\ :\ \sigma \in \{0,1\}^{\leq n}\ \&\ |s| \leq 1/2\}
\]
\[
 \cup\ 
\{\Psi_{\sigma,s,t}\ :\ \sigma \in \{0,1\}^{\leq n}\ \&\ |s| + |t| \leq 1\}
\]
\[
 \cup\ 
\{\Gamma_\sigma\ :\ \sigma \in \{0,1\}^{\leq n}\}.
\]
We then let $T = \bigcup_n T_n$.

We divide the rest of the proof into the following lemmas.

\begin{lemma}\label{lm:Tn.sat}
$T_n$ is satisfiable.
\end{lemma}

\begin{proof}
Let $\mathcal{A}$ denote the interpretation of $\LBL$ induced by the Banach lattice $L^p[0,1]$. 

By Theorem \ref{thm:lr.Lp}, there is an isometric embedding $S$ of $\ell^r$ into $L^p[0,1]$.  Set $f_j = S(e_j)$ (where $\{e_0, e_1, \ldots\}$ is the standard basis for $\ell^r$).  

When $\sigma \in \{0,1\}^n$, let $\nu(\sigma)$ be the number represented in base 2 by $\sigma$; that is $\nu(\sigma) = \sum_{j <n} \sigma(j)2^{-j}$.  
Let $\mathcal{A}'$ extend $\mathcal{A}$ to $\mathcal{L}$ by setting
\[
c_\sigma^{\mathcal{A}'} = \left\{
\begin{array}{cc}
\zerovec^\mathcal{A} & |\sigma| > n\\
2^{-n/r}f_{\nu(\sigma)} & |\sigma| = n\\
c_{\sigma\cat(0)}^{\mathcal{A}'} + c_{\sigma\cat(1)}^{\mathcal{A}'} & |\sigma| < n\\
\end{array}
\right.
\]
We show that $\mathcal{A}' \models T_n$.  
It immediately follows from the definition of $\mathcal{A}'$ that $\mathcal{A}' \models \Phi_{\sigma,s}$ and 
$\mathcal{A}' \models \Gamma_\sigma$.  \\

To see that $\mathcal{A}' \models \Psi_{\sigma,s,t}$, we first observe that 
\begin{equation}
c_\sigma^{\mathcal{A}'} = \sum_{|\tau| = n, \tau \sqsupseteq \sigma} c_\tau^{\mathcal{A}'}.\label{eqn:sum.leaves} 
\end{equation}
whenever $|\sigma| \leq n$.  

We now note that $\langle f_j \rangle_{j \in A}$ and $\langle f_j \rangle_{j \in B}$ are $L^r$-formally disjointly supported whenever $A,B \subseteq \N$ are disjoint.  For, if $A,B \subseteq \N$ are disjoint, then $\langle e_j \rangle_{j \in A}$, $\langle e_j \rangle_{j \in B}$ are disjointly supported.  
Since $S$ is an isometric embedding it follows that 
$\langle f_j \rangle_{j \in A}$ and $\langle f_j \rangle_{j \in B}$ are $L^r$-formally disjointly supported.  

We now claim that if $\sigma, \tau \in 2^{\leq n}$ are incomparable, then $c_\sigma^{\mathcal{A}'}$ and 
$c_\tau^{\mathcal{A}'}$ are $L^r$-formally disjointly supported.  For, suppose $\sigma, \tau \in \{0,1\}^{\leq n}$ are incomparable.  By the definition of $\nu$, 
$\{\nu(\tau')\ :\ \tau' \sqsupseteq \sigma\ \wedge\ |\tau'| = n\}$ and
$\{\nu(\tau')\ :\ \tau' \sqsupseteq \tau \ \wedge\ |\tau'| = n\}$ are disjoint.
Thus, 
$\langle \{c_{\tau'}^{\mathcal{A}'}\ :\ \tau' \sqsupseteq \sigma\ \wedge\ |\tau'| = n\} \rangle$ 
and $\langle \{c_{\tau'}^{\mathcal{A}'}\ :\ \tau' \sqsupseteq \tau \ \wedge\ |\tau'| = n\} \rangle$
are $L^r$-formally disjointly supported.  It then follows from (\ref{eqn:sum.leaves}) that $c_\sigma^{\mathcal{A}'}$ and 
$c_\tau^{\mathcal{A}'}$ are $L^r$-formally disjointly supported. 

Thus, $\mathcal{A}' \models \Psi_{\sigma,s,t}$.  
\end{proof}

Hence, by the compactness theorem, $T$ is satisfiable.  Let $\mathcal{A} \models T$.  By Theorem \ref{thm:ext.Lp},
there is an abstract real $L^p$ space $\mathcal{B}$ so that there is an isomorphic embedding $F$ of $\mathcal{A}$ into $\mathcal{A}_\mathcal{B}$.  By the Theorem \ref{thm:kakutani} we can (and do) assume $\mathcal{B}$ is a real $L^p$ space.  
Let $\phi(\sigma) = F(c_\sigma^\mathcal{A})$, and let $\mathcal{B}_0$ denote the closed linear span of $\ran(\phi)$.  

\begin{lemma}\label{lm:form.disint}
$\phi$ is an $L^r$-formal disintegration of $\mathcal{B}_0$.  
\end{lemma}

\begin{proof}
Since $\mathcal{A} \models \Psi_{\sigma,1/2}$, and since $F$ is an isomorphism, it follows that $\phi$ is summative.

We now claim that $\phi(\sigma\cat(0))$ and $\phi(\sigma\cat(1))$ are $L^r$-formally disjointly supported. 
To see this, let $s,t \in \R$, and let
\begin{eqnarray*}
M & =& \max\{|s|,|t|,1\}\\
s' & = & s/(2M)\\
t' & = & t/(2M)
\end{eqnarray*}
Since $\mathcal{A} \models \Psi_{\sigma,s',t'}$, 
\[
\norm{s' c_{\sigma\cat(0)} + t' c_{\sigma\cat(1)}}^r = 
|s'|^r\norm{c_{\sigma\cat(0)}}^r + |t'|^r\norm{c_{\sigma\cat(1)}}^r.
\]
Since $F$ is an isomorphism, it follows that 
\[
\norm{s' x_{\sigma\cat(0)} + t' x_{\sigma\cat(1)}}^r = 
|s'|^r\norm{x_{\sigma\cat(0)}}^r + |t'|^r\norm{x_{\sigma\cat(1)}}^r.
\]
Now multiply by $(2M)^r$.   It follows that $\phi$ is separating.
\end{proof}

\begin{lemma}\label{lm:isom.Lr}
$\mathcal{B}_0$ is isometrically isomorphic to $L^r([0,1]; \R)$.  
\end{lemma}

\begin{proof}
Let:
\begin{eqnarray*}
J_\emptyset & = & [0,1]\\
J_{\sigma\cat(0)} & = & [\min(I_\sigma), \frac{1}{2}(\min(I_\sigma) + \max(I_\sigma))]\\
J_{\sigma\cat(1)} & = & [\frac{1}{2}(\min(I_\sigma) + \max(I_\sigma)), \max(I_\sigma)]
\end{eqnarray*}
Let $\psi(\sigma) = \Ind_{J_\sigma}$.  Thus, $\psi$ is an $L^r$-formal disintegration of $L^r([0,1];\R)$.  Since $\mathcal{A} \models \Gamma_\sigma$, it follows that the identity map on $\{0,1\}^*$ is an isomorphism of $\psi$ with $\phi$.  
So, by Theorem \ref{thm:lifting}
this isomorphism induces an isometric isomorphism of 
$L^r([0,1]; \R)$ with $\mathcal{B}_0$.  
\end{proof}

\begin{lemma}\label{lm:B0.embeds}
$\mathcal{B}_0$ isometrically embeds in $L^p([0,1]; \R)$.  
\end{lemma}

\begin{proof}
Let $\mathcal{B} = L^p(\Omega; \R)$.  It follows from Lemma \ref{lm:sep.sub} that there is a separable subspace $\Omega_0$ of $\Omega$ so that $\mathcal{B}_0 \subseteq L^p(\Omega_0; \R)$.  
Since $\Omega_0$ is separable, so is $L^p(\Omega_0; \R)$.  Thus, by Theorem \ref{thm:sep.embed}, $L^p(\Omega_0; \R)$ isometrically embeds in $L^p([0,1]; \R)$.
\end{proof}

The complex case is not much different.  One adds two families $\{c_\sigma\}_{\sigma \in \{0,1\}^*}$, $\{d_\sigma\}_{\sigma \in \{0,1\}^*}$ of new and distinct constant symbols to $\LLpC$.  Then, to $\TLpC$ one adds sentences that altogether say $\sigma \mapsto (c_\sigma, d_\sigma)$ is a formal $L^r$-disintegration.  

\section{Conclusion}\label{sec:conclusion}

Our first finding is a new proof of Theorem \ref{thm:BDCK} by means of the compactness theorem of continuous logic.  
When forming $\LB$ and $\LBL$, there is no harm in restricting scalars to rational scalars.  
Thus, we only need the compactness theorem for countable languages.  
I.Goldbring has essentially given a constructive proof of the compactness theorem for continuous logic by adapting the Henkin construction to the continuous setting \cite{Goldbring.2018}; see also \cite{Calvert.2011}.  Thus the given proof does not make use of ultraproducts or the axiom of choice either explicitly or implicitly.  

We have also characterized complex $L^p$ spaces in the framework of Banach lattices and 
we have used this result to show how to represent complex $L^p$ spaces in the framework of continuous logic.  The combination of these results, together with the required material from the theory of complex $r$-stable random variables, then allows us to easily extend Theorem \ref{thm:BDCK} to the complex case; i.e. Theorem \ref{thm:main}.

\section*{Acknowledgement}

The author was supported in part by Simons Foundation Grant \# 317870.  

\def\cprime{$'$}
\providecommand{\bysame}{\leavevmode\hbox to3em{\hrulefill}\thinspace}
\providecommand{\MR}{\relax\ifhmode\unskip\space\fi MR }
\providecommand{\MRhref}[2]{%
  \href{http://www.ams.org/mathscinet-getitem?mr=#1}{#2}
}
\providecommand{\href}[2]{#2}

\section*{Appendix}

The following appears to be a matter of folklore in the theory of Banach lattices.

\begin{theorem}\label{thm:compl.Lp}
The complexification of $L^p(\Omega; \R)$ is $L^p(\Omega; \C)$. 
\end{theorem}

We give a proof which was suggested by J. Gl\"uck whose thesis contains a number of results on complex Banach lattices \cite{Glueck.2016}.  

A Banach lattice is \emph{Dedekind complete} if each nonempty set of vectors that is bounded above has a supremum.  A Banach lattice is \emph{super Dedekind complete} if it has the property that whenever $X$ is a nonempty set of vectors that is bounded above, there is a countable $D \subseteq X$ so that $\sup X = \sup D$.   It is well-known that $L^p$ spaces are Dedekind complete and that $\sigma$-finite $L^p$ spaces are super Dedekind complete.  

Now, let $f \in L^p(\Omega; \C)$, and let $g(t) = |f(t)|$.  It suffices to show that 
$g = \sup_\theta \Re(e^{i \theta} f)$.  Let $h = \sup_\theta \Re(e^{i \theta} f)$.  We first note that for each complex number $z$, 
$z = \sup_\theta \Re(e^{i \theta} z)$.   It follows that $g \geq h$.  Suppose $g_0 \geq h$.  
Since the support of $f$ is a $\sigma$-finite set, we can assume $\Omega$ is $sigma$-finite.  Thus, $L^p(\Omega; \R)$ is super Dedekind complete, and so there is a countable set of reals $D$ so that $h = \sup_{\theta \in D} \Re(e^{i \theta} f)$; we can additionally assume that $D$ is dense.  
Now, for each $\theta \in D$, $g_0(t) \geq \Re(e^{i \theta}f(t))$ a.e..  Since $D$ is countable, 
it follows that $g_0(t) \geq \sup_{\theta \in D} \Re(e^{i \theta} f(t))$ a.e..  Thus, 
since $D$ is dense, $g_0(t) \geq g(t)$ a.e..  We conclude that $g = h$.  

\end{document}